\documentclass[11pt]{amsart}

\usepackage{amssymb,amsthm,amsmath,hyperref,enumitem}
\usepackage[alphabetic]{amsrefs}

\usepackage{xypic}

\usepackage[margin=2.9cm]{geometry}

\newtheorem{theorem}{Theorem}[section]
\newtheorem{corollary}[theorem]{Corollary}
\newtheorem{lemma}[theorem]{Lemma}
\newtheorem{proposition}[theorem]{Proposition}

\theoremstyle{remark}
\newtheorem{remark}[theorem]{Remark}
\newtheorem*{remark*}{Remark}
\newtheorem{example}[theorem]{Example}

\numberwithin{equation}{section}

\newcommand{\N}{\mathbb{N}}
\newcommand{\Z}{\mathbb{Z}}

\newcommand{\C}{\mathbb{C}}

\newcommand{\Cu}{\mathrm{Cu}}

\newcommand{\her}{\mathrm{her}}

\renewcommand{\epsilon}{\varepsilon}
\renewcommand{\leq}{\leqslant}
\renewcommand{\geq}{\geqslant}

\title{Nuclear dimension and sums of commutators}
\author{Leonel Robert}
\address{
Department of Mathematics
University of Louisiana at Lafayette
Lafayette, USA.
}
\email{lrobert@louisiana.edu}
\subjclass[2010]{46L05}

\begin{document}
\begin{abstract}
The problem of expressing a selfadjoint element that is zero on every bounded trace as a finite sum (or a limit of sums)  of commutators is investigated in the setting of C*-algebras of finite nuclear dimension. Upper bounds -- in terms of the nuclear dimension of the C*-algebra -- are given for the number of commutators needed in these sums. An example is given of a simple, nuclear C*-algebra (of infinite nuclear dimension) with a unique tracial state and with elements that vanish on all bounded traces and yet are ``badly" approximated by finite sums of commutators. Finally, the same problem is investigated on (possibly non-nuclear) simple unital C*-algebras assuming suitable regularity properties
in their Cuntz semigroups. 
\end{abstract}

\maketitle

\section{Introduction}
Let $A$ be a C*-algebra. Let $A_0$ denote the intersection of the kernels of all bounded traces on $A$ further intersected with
 the selfadjoint elements of $A$. A theorem of Cuntz and Pedersen says that $A_0$ agrees with the set of elements of the form $\sum_{i=1}^\infty [x_i^*,x_i]$, with $ x_i\in A$, and where the infinite sum is norm convergent. 
A line of research pursued by several authors -- \cite{fack}, \cite{thomsen}, \cite{pop}, \cite{marcoux}, \cite{kaftal}, \cite{ng1} -- has been to find conditions on $A$ under which the elements of $A_0$ are expressible as finite, rather than infinite, sums of self-commutators (i.e., commutators of the form $[x^*,x]$). In this paper we explore the link between this question and the nuclear dimension of C*-algebras. The latter is a notion of dimension for C*-algebras due to Winter and Zacharias, which extends the covering dimension of topological spaces to the non-commutative setting (see \cite{winter-zacharias}).
We prove the following theorem:
\begin{theorem}\label{thm:nuclearcommutators}
Let $A$ be a C*-algebra of nuclear dimension $m\in \N$. 
If $a\in A_0$  then $a$ is a limit of sums of $m+1$ commutators of the form $[x^*,x]$,
with $\|x\|^2\leq 2\|a\|$. 
\end{theorem}
In \cite{fack}, Fack devised a technique  -- further developed in \cite{thomsen} and \cite{ng1} -- to show that elements of $A_0$, under suitable hypothesis, are finite sums of commutators. We generalize Fack's technique to prove the following theorem: 
\begin{theorem}\label{thm:fack}
Let $m\in \N$. There exists $M=O(m^3)$ such that if  $A$ is a  unital C*-algebra of nuclear dimension $m\in \N$, with no  finite dimensional representations and no simple purely infinite quotients, then  each element of $A_0$ is exactly the sum of $M$ self-commutators.
\end{theorem}

The hypothesis of no simple purely infinite quotients is most likely unnecessary, but our present method for constructing ``large" orthogonal positive elements in $A$ requires it. 

In \cite{marcoux}, Marcoux gives a ``reduction argument" for the number of commutators
needed for expressing  elements of $A_0$. Ng further developed this method in   \cite{ng1} and \cite{ng2}, showing for example that a unital embedding of the Jiang-Su algebra 
is often sufficient to make Marcoux's reduction  work.  Theorem \ref{thm:nuclearcommutators} above, combined with \cite[Theorem 3.2]{ng2}, yields the following corollary:  if $A$ is a unital C*-algebra of finite nuclear dimension, and the Jiang-Su algebra $\mathcal Z$ embeds  unitally in $A$, then every element of $A_0$ is the sum of three commutators (see Corollary \ref{cor:threecommutators}). In Section \ref{smallnumber} we relax the assumption of having an embedding of the Jiang-Su algebra by asking for an embedding of a dimension drop algebra $\mathcal Z_{n,n+1}$:
\begin{theorem}\label{thm:reduction}
Let $A$ be a unital C*-algebra of nuclear dimension $m\in \N$,  with no  finite dimensional representations and no purely infinite quotients.
Suppose also that $A$ admits a unital embedding of a  dimension drop algebra $\mathcal Z_{n,n+1}$ for $n$ large enough (depending only on $m$). Then each element of $A_0$ is the sum of three commutators.
\end{theorem}

An example of Pedersen and Petersen shows that  for each $m\in \N$ there exists a homogeneous C*-algebra $A$ of nuclear dimension $2m$ and $a\in A_0$ such that $a$ is not a limit of sums of $m+1$ self-commutators (see \cite[Lemma 3.5]{pedersen-petersen} and \cite[Theorem 2.1]{tristan-farah}). In Section \ref{sec:example} we construct a variety of examples of the same nature. We then use a construction \`a la Villadsen (of the second type) to prove the following theorem:

\begin{theorem}\label{thm:example}
There exists a simple unital AH C*-algebra $A$ with a unique tracial state such that for each $m\in \N$ there exists a contraction $a_m\in A_0$ such that for all $x_1,x_2,\dots,x_m\in A$ we have
$\|a_m-\sum_{i=1}^m [x_i^*,x_i]\|\geq 1$.
\end{theorem}
This theorem answers \cite[Question 2.3]{tristan-farah}.

It was demonstrated in \cite{marcoux} that the property of strict comparison of projections could be used to express elements of $A_0$, with $A$ of real rank zero, as sums of a small number of commutators. Strict comparison of positive elements and a unique tracial state were exploited in \cite{ng2} for the same purpose. 
Since such regularity properties continue to  exist beyond nuclear C*-algebras, they provide a path for extending the previous results  without assuming nuclearity. Almost divisibility in the Cuntz semigroup is another key property, satisfied for example by all $\mathcal Z$-stable C*-algebras. In \cite{winter}, Winter calls ``pure C*-algebra" one with both strict comparison of positive elements and almost unperforated Cuntz semigroup. In the last section of this paper we prove the following theorem:

\begin{theorem}\label{beyond}
Let $A$ be a separable simple, unital, pure C*-algebra of stable rank one. Assume also that bounded 2-quasitraces on $A$ are traces (e.g., $A$ is exact). Then each element of $A_0$ is the sum of three commutators.
\end{theorem}
The previous theorem allows us to deal with the case of separable simple unital exact $\mathcal Z$-stable C*-algebra completely: R\o rdam's dichotomy \cite[Theorem 6.7]{rordam} says that a simple unital $\mathcal Z$-stable C*-algebra is either purely infinite or of stable rank one. In the purely infinite case it is known that two commutators suffice \cite{pop}. On the other hand, the previous theorem is applicable in the stable rank one case. Indeed, $\mathcal Z$-stability implies pureness (by \cite{rordam}) and a well known theorem of Haagerup says that all bounded 2-quasitraces on an exact C*-algebra are traces.

\section{Preliminaries}\label{prelims}
Let $A$ be a C*-algebra.
A commutator of $A$ is an element of the form $[x,y]:=xy-yx$ for some $x,y\in A$. We call self-commutator a commutator of the form $[z^*,z]$. By \cite[Theorem 2.4]{marcoux}, if a selfadjoint element is a sum of $m$ commutators, then it is a sum of $2m$ self-commutators.

We will make use of the arithmetic of Cuntz equivalence classes of positive elements. Let us recall it briefly (see \cite{ara-perera-toms} for more). Let $A_+$ denote the set of positive elements of $A$. Given $a,b\in A_+$, $a$ is said to be Cuntz smaller than $b$ if there exist $d_n\in A$ such that $d_n^*bd_n\to b$; $a$ is Cuntz equivalent to $b$ if $a$ is Cuntz smaller than $b$ and $b$ is Cuntz smaller than $a$. These relations we will be denoted by $a\precsim b$ and $a\sim b$ respectively.
We will make use of the following reformulation of Cuntz comparison: $a\precsim b$ if and only if for each $\epsilon>0$ there exists $x\in A$ such that $(a-\epsilon)_+=x^*x$ and $xx^*\in \mathrm{her}(b)$. Here and throughout the paper, $\her(b)$ denotes the hereditary subalgebra generated by $b$, i.e., $\overline{bAb}$ and $(t-\epsilon)_+=\max(t-\epsilon,0)$ for $t\geq 0$.

The Cuntz semigroup of $A$, denoted by $\Cu(A)$, is defined as the set of Cuntz equivalence classes of positive elements in $A\otimes \mathcal K$ (where $\mathcal K$ is the C*-algebra of compact operators on a separable Hilbert space). Given a positive element $a\in A\otimes \mathcal K$ we will denote its Cuntz equivalence class by $[a]$. The order on $\Cu(A)$ is defined as $[a]\leq [b]$
if $a\precsim b$ and addition is defined as $[a]+[b]=[a'+b']$, where $a\sim a'$, $b\sim b'$ and $a'\perp b'$ (it is always possible to find such $a',b'$ in $A\otimes \mathcal K$).

Let us now recall the definition of nuclear dimension, as introduced by Winter and Zacharias in \cite{winter-zacharias}. We start by defining order zero maps:
A completely positive contractive (c.p.c.) map between C*-algebras $\phi\colon A\to B$ is said to have order zero if it preserves orthogonality, i.e., $ab=0$ implies $\phi(a)\phi(b)=0$ for all $a,b\in A$. By \cite[Corollary 3.1]{winter-zacharias0}, in this case there exists a homomorphism  $\pi^\phi\colon A\otimes C_0(0,1]\to B$ such that $\pi^\phi(a)=\pi(a\otimes t)$, where $t\in C_0(0,1]$ is the identity function on $(0,1]$. Observe, for later use,  that this implies that $\phi([x^* ,x])=[y^*,y]$, with $y=\pi^\phi(x\otimes t^{1/2})$, and that $\phi$ maps $A_0$ to $B_0$.

The C*-algebra $A$ is said to have nuclear dimension at most $m\in \N$ if for each $k=0,1,\dots,m$ there exist nets of c.p.c. maps $\psi^k_\lambda\colon A\to C^k_\lambda$ and $\phi^k_\lambda\colon C^k_\lambda\to A$, where $C^k_\lambda$ is a finite dimensional C*-algebra and $\phi^k_\lambda$ has order zero for all $\lambda\in \Lambda$, such that 
\[
\sum_{k=0}^m \phi^k_\lambda\psi^k_\lambda(a)\to_{\lambda} a
\]
for all $a\in A$. It is a useful  fact that the maps $\psi^k_\lambda$ may be chosen to be asymptotically of order zero. That is, such that $\psi^k\colon A\to \prod_\lambda C^k_\lambda/\bigoplus_\lambda C^k_\lambda$ is an order zero map. We then have that
\begin{align}\label{factorization}
\xymatrix{
A\ar@{^{(}->}[rr]^{\iota=\sum_{k=0}^m\phi^k\psi^k}\ar[dr]_{\psi^k} & & A_\infty\\
& N^k\ar[ur]_{\phi^k}&
}
\end{align}
where $\iota$ is the diagonal inclusion, $N^k=\prod_\lambda C^k_\lambda/\bigoplus_\lambda C^k_\lambda$, $A_\infty=\prod_{\lambda} A/\bigoplus_\lambda A$, and the maps
$\phi^k\colon N^k\to A_\infty$ are defined entrywise using the maps $\phi^k_\lambda$, $\lambda\in \Lambda$.

\section{Proof of Theorem \ref{thm:nuclearcommutators}}\label{nuclearcommutators}
The following lemma is well known:
\begin{lemma}\label{lem:onecommutator}
Let  $A$ be either a  finite dimensional C*-algebra, a product $\prod_{i=1}^\infty A_i$, or the quotient $\prod_{i=1}^\infty A_i/\bigoplus_{i=1}^\infty A_i$, where each $A_i$ is a finite dimensional C*-algebra. If $a\in A_0$ then there exists $x\in A$
such that $a=[x^*,x]$ and $\|x\|^2\leq 2\|a\|$.  
\end{lemma}
\begin{proof}
If $A$ is finite dimensional, this is \cite[Lemma 3.5]{fack}. The result for $\prod_{i=1}^\infty A_i$ follows at once applying it to each entry. Finally, let  $A=\prod_{i=1}^\infty A_i/\bigoplus_{i=1}^\infty A_i$ and let $a\in A_0$. Then $a=\sum_{n=1}^\infty [x_n^*,x_n]$.
Let us choose lifts $(a_i)_{i=1}^\infty,(x_{n,i})_{i=1}^\infty\in \prod_{i=1}^\infty A_i$.
For each $\epsilon>0$, let $N\in \N$ be such that $\|a-\sum_{n=1}^N [x_n^*,x_n]\|<\epsilon$. Then $\|a_i-\sum_{n=1}^N [x_{n,i}^*,x_{n,i}]\|<\epsilon$ for all sufficiently large $i$.
Thus, for all such $i$, we find $y_{i}$ such that $\|a_i-[y_{i}^*,y_i]\|<\epsilon$, with $\|y_i\|^2\leq 2(\|a_i\|+\epsilon)^2$. The result follows letting $\epsilon=1,\frac 1 2,\frac 1 3, \dots$ and using a standard diagonal sequence argument.
\end{proof}

\begin{proof}[Proof of Theorem \ref{thm:nuclearcommutators}]
Let $a\in A_0$, and assume that $a=\sum_{i=1}^\infty [z_i^*,z_i]$. By \cite[Proposition 2.6]{winter-zacharias}, we may assume that $A$ is separable and $z_i\in A$ for all $i$.
Let $\psi^k$ and $\phi^k$ be maps as in \eqref{factorization}, with $k=0,1,\dots,m$ (and with $\Lambda=\N$). Then $\psi^k(a)\in (N^k)_0$, since $\psi^k$ is a c.p.c. map of order zero. By Lemma \ref{lem:onecommutator} applied to the C*-algebra $N^k$, $\psi^k(a)$ is expressible as a single self-commutator $[x_k^*,x_k]$, with $\|x_k\|^2\leq 2\|\psi^k(a)\|$. Since $\phi^k$ is also a c.p.c. map of order zero, each self-commutator $[x_k^*,x_k]$ gets mapped by $\phi^k$
to a self-commutator $[y_k^*,y_k]\in A_\infty$, with $y_k=\pi^{\phi^k}(x_k\otimes t^{1/2})$ (whence $\|y_k\|\leq \|x_k\|$). Thus,
\[
\iota(a)=\sum_{k=0}^{m} \phi^k\psi^k(a)=\sum_{k=0}^{m} [y_k^*,y_k],
\] 
from which the result clearly  follows.
\end{proof}

\begin{remark}\label{rem:nuclearcommutators}
If in Theorem \ref{thm:nuclearcommutators} we relax the requirement of using self-commutators, 
and simply seek to express $a\in A_0$ as a limit of  sums of $m+1$ commutators $[x,y]$, then we can arrange for all the commutators to satisfy $\|x\|\cdot \|y\|\leq  \|a\|$. Let us show this: By the method used to prove Theorem \ref{thm:nuclearcommutators}, this boils down to showing that for $a\in A_0$, with $A$ a matrix algebra, we can find $x,y\in A$ with $a=[x,y]$ and $\|x\|\cdot \|y\|\leq \|a\|$. To prove this we may  assume that $a$ is a diagonal matrix. Furthermore, conjugating by a suitable permutation matrix, we can arrange for the partial sums of the diagonal entries of $a$ to be in $[-\|a\|,\|a\|]$. We can then use a commutator formula similar to equation \eqref{simplecommutator} below.
\end{remark}

\section{Finitely many commutators}\label{finitelymany}
In this section we prove Theorem \ref{thm:fack}.
\begin{lemma}\label{lem:trapecio}
Let $a,b\in A_+$ be such that  
\begin{align}\label{LKineq}
L[a]\leq (L-1)[(a-\epsilon)_+]+K[b]
\end{align} 
for some $L,K\in \N$ and $\epsilon>0$. Then for each $x\in \mathrm{her}((a-\epsilon)_+)$ we have
\[
x=\sum_{k=1}^{L(L+K-1)}[x_k,y_k]+z,
\]
for some $z\in \mathrm{her}(b)$ such that $\|z\|\leq K\|x\|$ and $x_k,y_k\in A$ such that $\|x_k\|\cdot \|y_k\|\leq \|x\|$ for all $k$.
\end{lemma}
\begin{proof}
We may assume without loss of generality that $a$ and $b$ are contractions. 
Let us set 
\[
c:=(a-\epsilon)_+\otimes 1_{L-1})\oplus (b\otimes 1_K)\in M_{L+K-1}(A).
\]
The relation \eqref{LKineq} implies that there exists an $L\times (L+K-1)$ matrix $V$ with entries in $A$ such that $g_{\epsilon/2}(a)\otimes 1_L=V^*V$ and 
$VV^*\in \her(c)$.
Adding over the main diagonal of $V^*V$ we get 
\[
L\cdot g_{\epsilon/2}(a)=\sum_{j=1}^L\sum_{i=1}^{L+K-1} v_{i,j}^*v_{i,j}.
\]
On the other hand, using that  $VV^*\in\mathrm{her}(b)$, and adding over the first $L-1$ diagonal terms of $VV^*$ we get that 
\begin{align}\label{ineq}
\sum_{i=1}^{L-1}\sum_{j=1}^L  v_{i,j}v_{i,j}^*\leq (L-1) \cdot g_{\epsilon/2}(a).
\end{align}
Finally, looking at  the remaining $K$ diagonal terms of $VV^*$ we get $v_{i,j}v_{i,j}^*\in\her(b)$ for all $i\geq L$ and all $j$.

Let $\Phi\colon \her((a-\epsilon)_+)\to \her((a-\epsilon)_+)$ be defined by
\[
\Phi(x)=\frac{1}{L}\sum_{i=1}^{L-1}\sum_{j=1}^L v_{i,j}yv_{i,j}^*.
\]
From \eqref{ineq} we deduce  that $\|\Phi\|\leq \frac{L-1}{L}$. Thus, $\mathrm{Id}-\Phi$ is invertible with $\|(1-\Phi)^{-1}\|\leq L$. Let
$y\in \her((a-\epsilon)_+)$ be such that $x=y-\Phi(y)$, with $\|y\|\leq L\|x\|$. Notice that $g_{\epsilon/2}(a)y=y$. Then
\begin{align*}
x &=y-\Phi(y)\\
&=g_{\epsilon/2}(a)y-\frac{1}{L}\sum_{i=1}^{L-1}\sum_{j=1}^L v_{i,j}yv_{i,j}^*\\
 &=\frac{1}{L}\sum_{i=1}^{L+K-1}\sum_{j=1}^L v_{i,j}^*v_{i,j}y-\frac{1}{L}\sum_{i=1}^{L-1}\sum_{j=1}^L v_{i,j}yv_{i,j}^*\\
&=\sum_{i=1}^{L+K-1}\sum_{j=1}^L \frac{1}{L}[v_{i,j}^*,v_{i,j}y]+z,
\end{align*}
with 
\[
z=\frac{1}{L}\sum_{i=L}^{L+K-1}\sum_{j=1}^Lv_{i,j}yv_{i,j}^*\in \mathrm{her}(b).\] 
Observe that there are $L(L+K-1)$ commutators $\frac{1}{L}[v_{i,j}^*,v_{i,j}y]$, that  $\frac{1}{L}\|v_{i,j}^*\|\cdot \|v_{i,j}y\|\leq \|x\|$ for all $i,j$, and that
\[
\|z\|\leq \Big \|\sum_{i=L}^{L+K-1}\sum_{j=1}^Lv_{i,j}v_{i,j}^*\Big\|\|x\|\leq K\|x\|,
\]
as desired.
\end{proof}

The following proposition is an adaptation of Fack's method from \cite{fack} that uses positive elements instead of projections. (We will only use it below with $L=1$.)
\begin{proposition}\label{prop:fack}
Let $A$ be a  C*-algebra. Suppose that 
\begin{enumerate}
\item
there exists a sequence of  positive elements $(e_i)_{i=0}^\infty$ in $A$ and $\epsilon_i>0$ such that   $e_i\perp  e_j$ for all $i,j\geq 1$,  
\begin{align}\label{trapecio}
L[e_i]\leq (L-1)[(e_i-\epsilon_i)_+)]+K[(e_{i+1}-\epsilon_{i+1})_+]
\end{align}
for   all $i\geq 0$, and
\item
there exist  $M,\overline M\geq 1$ such that for each hereditary subalgebra  $B\subseteq A$  if  $b\in B_0$ then $b$ is limit of elements of the form  $\sum_{i=1}^M [x_i,y_i]$, with
$x_i,y_i\in B$  and  $\|x_i\|\cdot \|y_i\|\leq \overline M\|b\|$ for all $i=1,2,\dots,M$.
\end{enumerate}
Then each element  $z\in \her((e_0-\epsilon_0)_+)\cap A_0$ is the sum of $L(L+K-1)+\max(M,L(L+K-1))$ commutators $[x,y]$, with  $\|x\|\cdot \|y\|\leq K\overline M\|z\|$.
\end{proposition}
\begin{proof}
Let $N=L(L+K-1)$ and $B^{(i)}=\mathrm{her}(e_i-\epsilon_i) $ for $i\geq 0$. Note that, by assumption,  $B^{(i)}\perp B^{(j)}$ for all $i,j\geq 1$ with $i\neq j$.

Let us choose a sequence $\delta_1,\delta_2,\dots$,  of positive real numbers tending to 0.

Let $z\in B^{(0)}\cap A_0$. By Lemma \ref{lem:trapecio} and \eqref{trapecio}, we can find elements $x_k^{(1)},y_k^{(1)}\in A$, with $k=1=,\dots,N$, such that
$
z=\sum_{k=1}^N [x_k^{(1)},y_k^{(1)}]+z_1,
$ where $z_1\in B^{(1)}$. 
Applying the assumption (ii) in the hereditary subalgebra $B^{(1)}$, we get  
$
z_1=\sum_{k=1}^M[\tilde x_k^{(1)} ,\tilde y_k^{(1)} ]+\tilde z_1,
$ where the elements $\tilde x_k^{(1)},\tilde y_k^{(1)}\in B_k^{(1)}$ are such that 
$\|x_k^{(1)}\|\cdot \|y_k^{(1)}\|\leq \overline M\|z_1\|$ for all $k$, and
where $\tilde z_1\in B^{(1)}$ is such that $\|\tilde z_1\|<\delta_1$. 
Now let us apply Lemma \ref{lem:trapecio} to $\tilde z_1$, and get
$
\tilde z_1 =\sum_{k=1}^M[x_k^{(2)},y_k^{(2)}]+z_2,
$
with $z_2\in B^{(2)}$.
Again by (ii), applied in the hereditary $B^{(2)}$, we have that
$
z_2=\sum_{k=1}^M[\tilde x_k^{(2)},\tilde y_k^{(2)}]+\tilde z_2,
$
with $\|\tilde z_2\|<\delta_2$.  
Let us continue applying Lemma \ref{lem:trapecio}, and then (ii), ad infinitum. We  get
\[
z=S_1+T_1+S_2+T_2+S_3+T_3+\cdots
\]
where
$S_i=\sum_{k=1}^N [x_k^{(i)},y_k^{(i)}]$ and $T_i=\sum_{k=1}^M [\tilde x_k^{(i)},\tilde y_k^{(i)}]$
for all $i$.
Notice that $S_i\in \overline{B^{(i-1)}+B^{(i)}}$  and 
$T_i\in B^{(i)}$ for all $i\geq 1$. Let us define
\begin{align*}
u&=(T_1+S_2)+(T_2+S_4)+(T_3+S_6)+\cdots,\\
v&=S_3+S_5+\cdots,
\end{align*}
so that $z=S_1+u+v$. Now notice  that if a sum of commutators  
$\sum_{i=1}^\infty [c_i,d_i]$ is such that $\|c_i\|\to 0$, $\|d_i\|\to 0$, and for each 
$i \neq j $ we have  $c_i\perp d_j$ (meaning that $c_i^*d_j=c_id_j=c_i^*d_j^*=0$),
$c_i\perp c_j$, and $d_i\perp d_j$, then it can be written as a single commutator $[c,d]$
with $c=\sum_{i=1}^\infty c_i$ and $d=\sum_{i=1}^\infty d_i$. Keeping this in mind,
and the orthogonality relations among the terms $S_{2},S_3,\dots$ and $T_1,T_2\dots$, we get that $u$ is a sum of $\max(N,M)$ commutators and $v$ a sum of $N$ commutators. The norm estimates for the terms in these sums of commutators can be checked easily form their construction. 
\end{proof}

Let us now show that Proposition \ref{prop:fack} is applicable to certain C*-algebras of finite nuclear dimension. By Theorem \ref{thm:nuclearcommutators}, the condition in Proposition \ref{prop:fack} (ii) holds in any C*-algebra of  nuclear dimension at most $M-1$ (since the nuclear dimension of a C*-algebra bounds the nuclear dimension of its hereditary subalgebras). What is missing is the construction of a sequence $(e_i)_{i=0}^\infty$ as in  Proposition \ref{prop:fack} (i). We accomplish this in Proposition \ref{prop:orthoseq} below. We first need a lemma, essentially taken from \cite{nuclearz}.

\begin{lemma}
Let $A$ be a C*-algebra of  nuclear dimension $m\in \N$  such that every representation of $A$ has dimension at least $2m+3$ and $A$ has no simple purely infinite quotients.   Let $c\in A_+$ be a strictly positive element and    $\epsilon>0$. Then there exist orthogonal positive elements $d_0,d_1\in A$ such that 
\begin{align*}
[(c-\epsilon]&\leq (2m+3)[d_0],\\
[c] &\leq (2m+2)(2m+3)[d_1].
\end{align*}
\end{lemma}
\begin{proof}
The construction of $d_0$ and $d_1$ with these properties is given in the proof of
\cite[Lemma 3.4]{nuclearz}.
\end{proof}

\begin{proposition}\label{prop:orthoseq}
Let $A$ be a  C*-algebra of nuclear dimension $m$, with no finite dimensional representations and no simple purely infinite quotients. Then for each strictly positive element $c\in A_+$ and $\epsilon>0$, there exist positive elements $e_0,e_1,\dots$,  and positive numbers $\epsilon_1,\epsilon_2,\dots$, with  $e_0=(c-\epsilon)_+$,   $e_i\perp e_j$ for  all $i,j\geq 1$, and  $[e_i]\leq K[(e_{i+1}-\epsilon_{i+1})_+]$ for all $i\geq 0$, where $K=O(m^3)$.
\end{proposition}
\begin{proof} 
Let us apply the previous lemma with $c$ and $\frac{\epsilon}{2}$, and let 
$d_0$ and $d_1$ be the resulting positive orthogonal elements.
From $[(c-\frac{\epsilon}{2})]\leq (2m+3)[d_0]$ we get that 
\[
[(c-\epsilon)_+]\leq (2m+3)[(d_0-\delta)_+]
\] for some $\delta>0$. Setting  $e_1=(d_0-\frac{\delta}{2})_+$
and $\epsilon_1=\frac{\delta}{2}$ we get  $[(c-\epsilon)]\leq [(e_1-\epsilon_1)_+]$, as desired.  \
The C*-algebra $\her(d_1)$ does not have finite dimensional representations (since $[c]\leq (2m+2)(2m+3)[d_1]$). Thus, we can apply the previous lemma to it. But first, let us choose $\delta_1>0$ such that 
\[
[(d_0-\frac{\delta}{2})_+]\leq (2m+2)(2m+3)[(d_1-\delta_1)_+].
\]
Now let us apply the previous lemma to $\her(d_1)$, with strictly positive element $d_1\in \her(d_1)_+$ and $\frac{\delta_1}{2}$. We get two positive orthogonal elements, $d_{1,0},d_{1,1}\in \her(d_1)$,
with $[(d_1-\frac{\delta_1}{2})_+]\leq (2m+3)[d_{1,0}]$. Thus, for some $\delta_2>0$, we have that $[(d_1-\delta_1)_+]\leq (2m+3)[(d_{1,0}-\delta_2)_+]$, and so
\[
[(d_0-\frac\delta 2)_+]\leq (2m+2)(2m+3)^2[(d_{1,0}-\delta_2)_+].\] 
Setting
$e_2=(d_{1,0}-\frac{\delta_2}{2})_+$ and $\epsilon_2=\frac{\delta_2}{2}$, this is simply
$[e_1]\leq [(e_2-\epsilon_2)_+]$. Furthermore, as before with $\her(d_1)$,  $\her(d_{1,1})$ has no finite dimensional representations, so we can continue applying this algorithm to get the desired sequence. 
\end{proof}

\begin{proof}[Proof of Theorem \ref{thm:fack}]
This follows at once from Theorem \ref{thm:nuclearcommutators}, Proposition \ref{prop:fack}, and Proposition \ref{prop:orthoseq} (applied with $c=1$).
\end{proof}

\section{Small number of commutators}\label{smallnumber}
Theorem \ref{thm:nuclearcommutators} combined with \cite[Theorem 3.2]{ng2} yields at once the following
\begin{corollary}\label{cor:threecommutators}
Let $A$ be a unital C*-algebra of nuclear dimension $m$. Suppose that $A$ admits a
unital embedding of the Jiang-Su algebra $\mathcal Z$. If $a\in A_0$ then there exist
$x_i,y_i\in A$, with $i=1,2,3$, such that 
$a=[x_1,y_1]+[x_2,y_2]+[x_3,y_3]$.
\end{corollary}

Let us now prove Theorem \ref{thm:reduction}. We  follow \cite{ng1} and \cite{ng2}  closely.

\begin{lemma}\label{lem:matrices2}
Let $A$ be a C*-algebra, $c\in A_+$ and $\epsilon>0$. 
Let $b\in M_n(\her((c-\epsilon)_+))$ be such that 
\[
\sum_{i=1}^n b_{i,i}=\sum_{i=1}^n [x_i,y_i],
\]
with $x_i,y_i\in \her((c-\epsilon)_+)$ for all $i$. Then $b$ is the sum of two commutators in $M_n(A)$.
\end{lemma}
\begin{proof}
Let us write $b=b'+b''$, where 
\[
b'=\begin{pmatrix}
b_{1,1}-[x_1,y_1] & & &\\
&b_{2,2}-[x_2,y_2] & &\\
&&\ddots &
&&&b_{n,n}-[x_n,y_n]
\end{pmatrix},\quad
\]
Then $b'$ and $b''$ are commutators by \cite[Lemma 2.7]{ng1} and \cite[Lemma 2.8]{ng1} respectively.  The commutator formula for $b'$ is simply
\begin{align}\label{simplecommutator}
b'=\left[
\begin{pmatrix}
0 & s_1 &  & \\
& \ddots & \ddots &\\
&&&s_{n-1}\\
&&& 0
\end{pmatrix},
\begin{pmatrix}
0 &&&\\
e &\ddots &&\\
&\ddots &&\\
&&e&0
\end{pmatrix}
\right],
\end{align}
where $s_i=\sum_{j=1}^i (b_{j,j}-[x_j,y_j]) $ and $e\in C^*(c)_+$ is chosen such that $e(c-\epsilon)_+=(c-\epsilon)_+$.
\end{proof}

\begin{proof}[Proof of Theorem \ref{thm:reduction}]
We will make  use of the following relation among positive elements: $a\precsim_s b$ if $a=v^*v$ and $vv^*\in \her(b)$ for some $v\in A$ (equivalently $\overline{aA}$ embeds in $\overline{bA}$ as a Hilbert $A$-module).

 Let $n\geq 3$ (how much larger will be specified later).
Since $\mathcal Z_{n,n+1}$ embeds unitally in $A$, we can find positive elements
$e_i\in A_+$, with $i=1,2,\dots, n$ and $d\in A_+$ such that
\begin{enumerate}
\item[(1)]
$e_i\perp e_j$ for all $i\neq j$,
\item[(2)]
there exist $x_i$ such that $e_1=x_i^*x_i$, $x_ix_i^*=e_{i}$, for $i=2,\dots,n$,
\item[(3)] 
$1=d+\sum_{i=1}^n (e_i-\frac{1}{2})_+$, and $1-d\precsim_s (e_1-\frac{1}{2})_+$.
\end{enumerate} 
In fact, these elements may be found in $\mathcal Z_{n,n+1}$ and then moved to $A$
(see the proof of \cite[Lemma 4.2]{rordam}).

Let $a\in A_0$. Then 
\[
a=dz+(1-d)zd+(1-d)z(1-d).
\] 
Let $a'=dz+(1-d)zd$ and $f=d +ad^2a+(1-d)ad^2a(1-d)$. Then $a'\in \her(f)$.
Also,
\[
d\precsim_s (e_1-\frac 1 2)_+,\quad ad^2a\precsim_s d\precsim_s (e_2-\frac 1 2)_+,\quad  \hbox{and}\quad(1-d)ad^2a(1-d)\precsim_s d\precsim_s (e_3-\frac 1 2)_+.
\]
Since the $(e_i-\frac 1 2)_+$s   are pairwise orthogonal (and $n\geq 3$), 
we get that $f\precsim_s (\sum_{i=1}^ne_i-\frac 1 2)_+$.  Thus, by \cite[Lemma 2.5]{ng2},  $a'=[x,y]+a''$, with $z''\in \her((\sum_{i=1}^n e_i-\frac 1 2)_+)$. Hence, $a=[x,y]+b$, with $b\in \her(1-d)=\her((\sum_{i=1}^n e_i-\frac 1 2)_+)$.

Let us show that for $n$ large enough  Lemma \ref{lem:matrices2} can be applied to get that $b$ is the sum of two commutators.
From the properties (1)-(3) of the $e_i$s we have that
$\her(\sum_{i=1}^n e_i) \cong M_n(\her(e_1))$.
Thus, in order to be able to apply Lemma \ref{lem:matrices2}, we only need to show that
$\sum_{i=1}^n b_{i,i}$ is a sum of $n$ commutators, where $b=(b_{i,j})_{i,j}$
is regarded as an element of $M_n(\her((e_1-\frac 1 2)_+))$. Let us prove this.
Since $[1]\leq (n+1)[e_1]$, the C*-algebra $\her(e_1)$ has no finite dimensional representations and no simple purely infinite quotients. Also,  since $\her(e_1)$ is a full hereditary subalgebra of $A$,  every bounded trace on $\her(e_1)$ extends uniquely to a bounded trace on $A$. Hence,  $b\in A_0$ implies that $\sum_{i=1}^n b_i\in \her(e_1)_0$. Now, in the same way that we proved Theorem \ref{thm:fack}, we can deduce from Proposition \ref{prop:orthoseq}  and Proposition \ref{prop:fack} that $\sum_{i=1}^n b_{i,i}$ is a sum of $n$ commutators inside $\her((e_1-\frac 1 4)_+)$ for $n=O(m^3)$  depending solely on $m$. Then, by Lemma \ref{lem:matrices2},  $b$ is the sum of two commutators in $A$.
\end{proof}

\begin{remark}\label{rem:Zembed}
The hypotheses ``no finite dimensional representations and no simple purely infinite quotients" in Theorem \ref{thm:reduction} are only needed to ensure the existence of an infinite sequence of sufficiently large orthogonal positive elements  inside $\her(e_1)$, with $e_1$ as in the preceding proof. If instead we assume that $\mathcal Z$ embeds unitally in $A$, then the existence of such a sequence of  orthogonal elements can be obtained by more direct means. Indeed, in this case an embedding $\mathcal Z_{n,n+1}\hookrightarrow A$ may be chosen that factors through the embedding $\mathcal Z_{n,n+1}\stackrel{a\mapsto 1\otimes a}{\longrightarrow} \mathcal Z\otimes \mathcal Z\cong \mathcal Z$. In this way, we can arrange for an infinite sequence of positive orthogonal elements in $\her(e_1)$  satisfying \eqref{trapecio} (with $L=1$ and small $K$), which   makes Proposition \ref{prop:fack} applicable in $\her(e_1)$. This line of reasoning gives an alternative proof to Corollary \ref{cor:threecommutators}.
\end{remark}

\begin{remark}\label{rem:notreduced}
 In spite of the number of commutators being reduced -- say, from Theorem \ref{thm:nuclearcommutators} to Theorem \ref{thm:reduction} --  this is achieved at the expense of increasing the norms of the elements making up the commutators (see for example the commutator formula \eqref{simplecommutator} in the proof of Lemma \ref{lem:matrices2}). Although explicit estimates for the norms of these elements may be obtained, it does not seem that  if $A$ is a C*-algebra of nuclear dimension $m$ the quantity
\[
\sup_{\substack{a\in A_0\\ \|a\|\leq 1}} \sup_{\epsilon>0}\inf\{\,
\sum_{i=0}^m \|x_i\|\cdot \|y_i\|\mid \|a-\sum_{i=0}^m [x_i,y_i]\|<\epsilon\,
\, x_i,y_i\in A\}
\]
is  reduced by this method. (We know, by Remark \ref{rem:nuclearcommutators}, that it is at most $m+1$.)
\end{remark}

\section{Large number of commutators}\label{sec:example}
In this section we prove Theorem \ref{thm:example}.

Let $X$ be a compact Hausdorff space. It will be helpful in the sequel to bear in mind the correspondences between projections in $C(X,\mathcal K)$, finitely generated projective modules over $C(X)$, and vector bundles on $X$. For example, given a projection $p\in C(X,\mathcal K)$ then $E_p=\{(x,v)\mid v\in \ell_2, p(x)v=v\}$ is a vector bundle over $X$
whose continuous sections $P_p=\{s\in C(X,\ell_2)\mid p(x)s(x)=s(x)\}$ form a finitely generated projective module over $C(X)$. 

Let us fix an embedding of the Cuntz algebra $\mathcal O_2$ in $B(\ell_2)$
and use it to define direct sums of elements in $C(X,\mathcal K)$: if $f,g\in C(X,\mathcal K)$ then 
\[
f\oplus g:=v_1fv_1^*+v_2gv_2^*,
\] 
with $v_1,\in v_2\in B(\ell_2)$ isometries that generate $\mathcal O_2$.

In the sequel $1_X\in C(X,\mathcal K)$ denotes the rank one  projection constantly  equal to $e_{1,1}\in \mathcal K$.

\begin{lemma}
Let $p,q\in C(X,\mathcal K)$ be projections.
Suppose that $[p]\leq n[1_X]$ and that there exists $v\in q(C(X,\mathcal K))p$
that is full (i.e., $v(x)\neq 0$ for all $x\in X$). Then $[1_X]\leq n[q]$.
\end{lemma}
\begin{proof}
Without loss of generality, we may assume that $p\leq 1_X^{(n)}$, with $1_X^{(n)}(x):=1_n$ (the $n\times n$ identity matrix) for all $x\in X$. Hence, we can reduce to the case that $p=1_X^{(n)}$. Let $v_1,v_2,\dots,v_n$ denote the columns of $v$. They can be interpreted as sections of the vector bundle associated to $q$. Since they do not all simultaneously vanish, we get that $[1_X]\leq n[q]$.
\end{proof}

\begin{proposition}\label{prop:obstruction}
Let $p,q\in C(X,\mathcal K)$ be projections such that $[p]\leq n[1_X]$ and $[1_X]\nleq nm[q]$.
Let $a\in \her(p\oplus q)$ be of the form 
\[
\begin{pmatrix}
p & *\\
* & *
\end{pmatrix}.
\]
Then $a$ cannot be within a distance less than 1 of a sum of $m$ self-commutators in $\her(p\oplus q)$.
\end{proposition}
\begin{proof}
Suppose that we have elements
\[
x_i=
\begin{pmatrix}
a_i & b_i\\
c_i & d_i
\end{pmatrix},
\]
with $a_i\in \her(p)$, $b_i\in p(C(X,\mathcal K))q$, $c_i\in q(C(X,\mathcal K)p)$
and $d_i\in \her(q)$
for all $i=1,2,\dots,m$,  such that $\|a-\sum_{i=1}^m x_i^*x_i\|<1$. Then
\begin{align}\label{pcommutators}
\|p-\sum_{i=1}^m [a_i^*,a_i]-\sum_{i=1}^m (c_i^*c_i-b_ib_i^*)\|<1.
\end{align}
Let $v=(c_1^*-b_1)\oplus (c_2^*-b_2)\oplus \cdots \oplus (c_n^*-b_n)$. Then $v\in q^{\oplus m}C(X,\mathcal K)p$. By the previous lemma,
$v$ is not full, i.e., it vanishes at some point $x\in X$. Evaluating at $x$ in \eqref{pcommutators} we get
\[
\|p(x)-\sum_{i=1}^m [a_i^*(x),a_i(x)]\|<1,
\] 
with $a_i(x)\in \her(p(x))$ for all $i$.
This is impossible, since the unit  of a C*-algebra with a tracial state (namely, $\mathrm{her}(p(x))$) cannot be within a distance less than 1 from a sum of commutators.
\end{proof}

\begin{example}\label{ex:bott}
The previous proposition gives us a way of constructing examples of $a\in A_0$ not well approximated by small sums of commutators. For example, let $P\in M_2(C(S^2))$ denote a Bott projection over the 2-dimensional sphere (i.e., a rank one projection whose associated line bundle is the tautological line bundle of $\mathrm{CP}(1)\cong S^2$). Let $X=\prod_{i=1}^m (S^2)$ and   
$p=P^{\otimes m}\in M_{2^m}(C(X))\subset C(X,\mathcal K)$. If we identify 
$H^*(X)$ with $\Z[\alpha_1,\dots,\alpha_m]/(\alpha_i^2=0\mid i=1,\dots,m)$, then 
the Euler class of the vector bundle associated to $p^{\oplus m}$ is $\prod_{i=1}^m \alpha_i\neq 0$. Hence,  $[1_X]\nleq m[p]$. So the element $a\in\her(1_X\oplus p)$ given by
\[
a=\begin{pmatrix}1_X & 0 \\ 
0 & -p
\end{pmatrix}
\] 
is in $\her(1_X\oplus p)_0$ but not within a distance less than 1 of a sum of $m$ commutators.
\end{example}

\begin{proof}[Proof of Theorem \ref{thm:example}]
Let $(k_n)_{n=1}^\infty$ be an increasing sequence of natural numbers. 
We will use this sequence to  construct  the C*-algebra $A$ as an inductive. Then,  by letting  $(k_n)_{n=1}^\infty$ grow sufficiently fast, will show that $A$ has the desired properties.

 For each $i\in \N$, let $X_i=(S^2)^{k_i}$
and consider the rank one projection $P_i=P^{\otimes k_i}\in C(X_i,\mathcal K)$, 
where $P\in M_2(C(S^2))$ is a Bott projection. Let us now form the product
$Y_n=\prod_{i=1}^n X_i$ and consider, over this space, the projection
\[
p_n(y)= P_1(x_1)^{\oplus l_1}\oplus P_2(x_2)^{\oplus l_2}\oplus \cdots \oplus P_n(x_n)^{\oplus l_n},
\]
where $y=(x_1,x_2,\dots,x_n)\in Y_n$.  Let us set the numbers $l_n$ recursively  such that $l_1=1$ and $l_{n+1}=\mathrm{rank}(p_{n})$ for $n\geq 1$ (in fact, $l_n=2^{n-1}$ for $n\geq 2$).
Let $\phi_n\colon \her(p_n)\to \her(p_{n+1})$ be defined as follows:
\[
\phi_n(f)(y_n,x_{n+1})=f(y_n)\oplus (f(c_n)\otimes P_{n+1}(x_{n+1})),
\]
where $c_n\in Y_n$. In this formula, the term $f(c_n)\otimes P_{n+1}$
is regarded  as an element in $\her(P_{n+1}^{\oplus l_{n+1}})$ via the identifications
$\her(P_{n+1}^{\oplus l_{n+1}})\cong M_{l_{n+1}}(\C)\otimes \her(P_{n+1})$
 and $\her(p_n(c_n))\cong M_{l_{n+1}}(\C)$ (recall that $p_{n}$ has rank $l_{n+1}$).

It is known that choosing the points $c_n\in Y_n$ suitably, one can arrange for  the inductive limit C*-algebra $A:=\varinjlim (\her(p_n),\phi_n)$ to be simple and have a unique tracial state (see \cite{villadsen}; the uniqueness of tracial state comes from the fact that if $\sigma$ is a tracial state on $\her(p_{m+n})$) then $\sigma\circ \phi_{m,m+n}$ is an average of $\sigma$ and $2^n-1$ traces that don't depend on $\sigma$, coming from point evaluations).

Let us show that for a suitable choice of the sequence $(k_n)_{n=1}^\infty$ the inductive limit $A$ has the desired properties. It suffices, for each $m\in \N$, to find an element 
$a_m$ of norm 1 in  $\her(p_n)_0$ for some $n$,  that not only it is not within a distance less than 1 from any sum of $m$ commutators, but also this property is not destroyed by moving the element along the inductive limit.

Fix $m\in \N$. Let us assume that the numbers $k_1,\dots,k_m$ have been chosen. Let us find  $M_m\in \N$ such that $[p_m]\leq M_m[1_{Y_m}]$.
Now let us choose $k_{m+1}\in \N$ such that $k_{m+1}\geq  mM_ml_m$. An Euler class 
computation (as in Example \ref{ex:bott}) then shows that $[1_{Y_{m+1}}]\nleq mM_ml_m[P_{m+1}(x_{m+1})]$ in $\her(p_{m+1})$. It follows  by Proposition \ref{prop:obstruction}
that the element $a_m\in \her(p_{m+1})=\her(p_m\oplus P_{m+1}^{\oplus l_{m+1}}(x_{m+1}))$
given by
\[
a(y,x_{m+1})=\begin{pmatrix}
p_{m}(y) & \\
& -P_{m+1}^{\oplus l_{m+1}}(x_{m+1}),
\end{pmatrix}
\]
with $(y,x_{m+1})\in Y_m\times X_{m+1}=Y_{m+1}$,
is not within a distance less than 1 of a sum of $m$ commutators in $\her(p_{m+1})$.
Observe on the other hand that $a\in \her(p_{m+1})_0$.
Let us continue choosing the numbers $k_{m+2},k_{m+3},\dots$ in this way. Let us now consider  the image of $a\in \her(p_{m+1})$, as  defined above,  by the connecting map $\phi_{m+1,m+n}\colon \her(p_{m+1})\to \her(p_{m+n})$. Then
\[
\phi_{m+1,m+n}(a)=
\begin{pmatrix}
p_m & \\
& *
\end{pmatrix}
\]
where we regard $\her(p_{m+n})$ as $\her(p_{m}\oplus q_{m,n})$, with  
\[
q_{m,n}(y)=P_{m+1}(x_{m+1})^{\oplus l_{m+1}}\oplus \cdots \oplus P_{m+n}(x_{m+n})^{\oplus l_{m+n}},
\]
for $y\in Y_{n+m}$.
Again a routine Euler class computation shows that $[1_{Y_{m+n}}]\nleq mM_m[q]$. Hence, $\phi_{m,m+n}(a)$ is not within a distance less than 1 of a sum of $m$ commutators in $\her(p_{n+m})$, by Proposition \ref{prop:obstruction}.
\end{proof}

\section{Beyond nuclearity}
Here we prove Theorem \ref{beyond}. Let us first recall the properties of almost unperforation and almost divisibility in the Cuntz semigroup of a C*-algebra: 
The Cuntz semigroup $\Cu(A)$ is said to have almost divisibility if for each $k\in \N$, $[a]\in \Cu(A)$, and $\epsilon>0$, there
exists $[b]\in \Cu(A)$ such that $k[b]\leq [a]$ and $[(a-\epsilon)_+]\leq (k+1)[b]$. The property of strict comparison of positive elements of $A\otimes \mathcal K$
is equivalent to almost unperforation in the Cuntz semigroup. $\Cu(A)$ is said to be almost unperforated if $(k+1)[a]\leq [b]$ implies $[a]\leq [b]$ for all $k\in \N$
and $[a],[b]\in \Cu(A)$.

\begin{lemma}
Let $B\subseteq A$ be C*-algebras of stable rank one such that the inclusion map 
$B\stackrel{\iota}{\hookrightarrow} A$ induces an isomorphism 
$\Cu(\iota)\colon \Cu(B)\to \Cu(A)$ at the level of the Cuntz semigroups, and there exists $b\in B_+$ which is strictly positive in $A$. Then each selfadjoint element $a\in A$
is approximately unitarily equivalent to some selfadjoint element $a'\in B$ (with unitaries in the unitization of $A$).
\end{lemma}
\begin{proof}
We may assume that $a$ is a contraction. Let $a=a_+-a_-$. Let $\phi\colon C_0(0,1]\oplus C_0(0,1\to A$ be the *-homomorphism such that $\phi(t\oplus 0)=a_+$ and $\phi(0\oplus t)=a_-$. Then $\alpha=\Cu(\iota)^{-1}\circ \Cu(\phi)$
is a morphism in the category $\mathbf{Cu}$ from $\Cu(C_0(0,1]\oplus C_0(0,1])$ to $\Cu(B)$
such that $\alpha([t\oplus t])\leq [b]$. By the classification theorem \cite[Theorem 1]{ciuperca-elliott-santiago}, there exists $\psi\colon C_0(0,1]\oplus C_0(0,1]\to B$ such that $\Cu(\psi)=\alpha$.
Furthermore, the uniqueness part of this classification theorem implies that $\iota\circ \psi$ and $\phi$ are approximately unitarily equivalent. Thus, the desired result follows setting $a'=\psi(t\oplus 0)-\psi(0\oplus t)$.
\end{proof}

\begin{proof}[Proof of Theorem \ref{beyond}]
Let $A$ be as in the theorem. Since its Cuntz semigroup $\Cu(A)$ is both almost unperforated and almost divisible,  it can be computed to be
\[
\Cu(A)\cong \mathrm{V}(A)\sqcup \mathrm{SAff}(\mathrm{T}(A)).
\]
This computation is obtained in  \cite[Theorem 3.6]{brown-toms} and also in \cite[Theorem 5.6]{antoine-bosa-perera} (in both references exactness is only used to guarantee that bounded 2-quasitraces are traces). The additive structure of the ordered semigroup on the right hand side encodes the pairing $\lambda$ 
between tracial states $\tau\in \mathrm{T}(A)$ and Murray von-Neumann classes of projections $[p]\in \mathrm{V}(A)$, so that $\Cu(A)$ is in this case functorially equivalent to $(\mathrm V(A), \mathrm T(A),\lambda)$ (see \cite{antoine-dadarlat-perera-santiago}). On the other hand, Elliott has shown in \cite[Theorem 2.2]{elliott} that such data is exhausted by inductive limits of 
1-dimensional NCCW complexes with trivial $K_1$-group. Therefore, there exists a simple unital C*-algebra $B$, expressible as an inductive limit of 1-dimensional NCCW complexes with trivial $K_1$-group, such that $\Cu(B)\cong \Cu(A)$ with $[1]\mapsto [1]$. It follows by the classification theorem of \cite{nccw} that there exists a unital embedding of $B$ in $A$ inducing an isomorphism at the level of the Cuntz semigroups. 

Let $a\in A_0$. By the previous lemma, there exists $a'\in B$ unitarily equivalent to $a$. In particular, $\tau(a')=0$  for all bounded traces on $A$. But the fact that the inclusion
of $B$ in $A$ induces an isomorphism at the level of their Cuntz semigroups implies that
the restriction map $\tau\mapsto \tau|_B$ is a bijection from $\mathrm T(A)$ to $\mathrm T(B)$.
Thus, $a'\in B_0$. Since the nuclear dimension of $B$ is at most 1, we get that $a'$ is a limit of sums of two self-commutators $[x_0^*,x_0]+[x_1^*,x_1]$, with $\|x_i\|^2\leq 2\|a'\|$ for $i=0,1$. Thus, the same holds for $a$. Furthermore, $\mathcal Z$ embeds unitally in $B$ (which is in fact $\mathcal Z$-stable), whence also in $A$. Thus, the theorem follows from \cite[Theorem 3.2]{ng2} (alternatively, see  Remark \ref{rem:Zembed}).
\end{proof}

\end{document}